\documentclass[12pt, a4paper, reqno, oneside]{amsart}
\usepackage{amssymb}  
\usepackage{amsmath} 
\usepackage{amsthm}  
\usepackage{graphicx}
\usepackage{color}

\usepackage[cp1251]{inputenc}
\usepackage[english]{babel}
\usepackage{cite, verbatim}

\newtheorem{theorem}{Theorem}
\newtheorem*{theorem*}{Theorem}
\newtheorem{corollary}[theorem]{Corollary}
\newtheorem{lemma}{Lemma}
\numberwithin{lemma}{section}
\theoremstyle{remark}
\newtheorem{rem}{Remark}

\newcommand{\Aut}{\operatorname{Aut}}
\newcommand{\Out}{\operatorname{Out}}
\newcommand{\Soc}{\operatorname{Soc}}

\newcommand{\gexp}{\operatorname{exp}}

\def\ov{\overline}

\begin{document}

%

\title[Criterion of nonsolvability of a finite group]{Criterion of nonsolvability of a finite group\\ and recognition of direct squares of simple groups}

\author{Zh.\,Wang, A.\,V.\,Vasil'ev, M.\,A.\,Grechkoseeva, and A.\,Kh.\,Zhurtov}

\thanks{A.\,V.\,Vasil'ev was supported by National Natural Science Foundation of China (No.~12171126). A.\,V.\,Vasil'ev and M.\,A.\,Grechkoseeva were supported by RAS Fundamental Research Program, project  FWNF-2022-0002.}

\begin{abstract} The spectrum $\omega(G)$ of a finite group $G$ is the set of orders of its elements. The following sufficient criterion of nonsolvability is proved: if among the prime divisors of the order of a group $G$, there are four different primes such that $\omega(G)$ contains all their pairwise products but not a product of any three of these numbers, then $G$ is nonsolvable. Using  this result, we show that for $q\geqslant 8$ and $q\neq 32$, the direct square $Sz(q)\times Sz(q)$ of the simple exceptional Suzuki group $Sz(q)$ is uniquely characterized by its spectrum in the class of finite groups, while for $Sz(32)\times Sz(32)$, there are exactly four finite groups with the same spectrum. 


 \end{abstract}

\maketitle

\section*{Introduction}

Let $G$ be a finite group (in what follows, all groups are assumed to be finite). The set of prime divisors of its order and the set of its element orders are denoted by $\pi(G)$ and $\omega(G)$, respectively, the latter set for brevity we call the \emph{spectrum} of the group~$G$. Groups $G$ and $H$ are called \emph{isospectral} if $\omega(G)=\omega(H)$. A group $G$ is called \emph{recognizable by spectrum} if any group isospectral to $G$ is isomorphic to $G$. More details about recognition of simple and related groups by spectrum can be found in the recent survey~\cite{21Survey_arxiv}. The present paper concerns the recognition of direct squares of finite simple groups.

As a rule, the first step in proving the recognizability of nonsolvable groups is to establish that the groups isospectral to them are also nonsolvable, so we start with the following sufficient criterion of nonsolvability of a finite group, which may be of independent interest.

\begin{theorem}\label{th:criterion}
Let $G$ be a finite group. Suppose that there is a subset $\sigma(G)$ of $\pi(G)$ 
such that $|\sigma(G)|\geqslant 4$ and the following hold:
\begin{enumerate}
\item[(1)] $pq\in\omega(G)$ for all distinct $p,q\in \sigma(G);$
\item[(2)] $pqr\not\in\omega(G)$ for all pairwise distinct $p,q,r\in \sigma(G).$
\end{enumerate}

Then $G$ is nonsolvable. In particular, there is no such group of odd order.
\end{theorem}

\begin{rem}
Theorem~\ref{th:criterion} is wrong if $|\sigma(G)|=3$. It suffices to consider a $3$-primary group $G$ with $\pi(G)=\{p,q,r\}$, which is the direct product of two Frobenius groups whose kernels are a $p$-group and $q$-group, and the complements are $r$-groups.
\end{rem}

\begin{rem}
If $|\sigma(G)|\geqslant6$, then Theorem~\ref{th:criterion} follows from \cite[Theorem~1]{95Zha}. Indeed, let $G$ satisfy the conditions of Theorem~\ref{th:criterion} and $\sigma=\sigma(G)$. If $G$ is solvable, then it includes a Hall $\sigma$-sub\-group $H$ and, by Condition 2, the inequality $\alpha(H)\leqslant2$ holds. Here, as in~\cite{95Zha}, we denote by $\alpha(H)$ the maximum of $|\pi(\langle x\rangle)|$ for all elements $x$ of~$H$. Applying \cite[Theorem~1]{95Zha}, we obtain that $$|\pi(H)|\leq\alpha(H)(\alpha(H)+3)/2\leqslant5;$$ a contradiction. If $|\sigma(G)|\leqslant5$, then, as examples from~\cite{94Kel,95Zha} show, it is impossible to derive the conclusion of Theorem~\ref{th:criterion} applying Condition~2 only.
\end{rem}

\begin{rem} If $\sigma(G)$ satisfies the additional condition:
$p$ does not divide $q-1$ for all $p,q\in \sigma(G)$, then the conclusion of Theorem~\ref{th:criterion} was proved in~\cite[Proposition~1]{21GorMas}.
\end{rem}

Before returning to recognition, we introduce for convenience one more concept related to element orders: the prime graph $GK(G)$ of a group $G$ is a graph with vertex set $\pi(G)$, in which two vertices $p $ and $q$ are adjacent if and only if $p\neq q$ and $pq\in\omega(G)$. Denote by $t(G)$ the maximum number of pairwise nonadjacent vertices in~$GK(G)$. If the group $L$ satisfies the condition $t(L)\geqslant4$, then it is easy to see that the direct square $L\times L$ satisfies the conditions of Theorem \ref{th:criterion}, and we get

\begin{corollary}\label{c:sqr} Let $L$ be a finite group and $t(L)\geqslant 4$. If $G$ is a finite group such that $\omega(G)=\omega(L\times L)$, then $G$ is nonsolvable.
 \end{corollary}

\begin{rem} The condition $t(L)\geqslant 4$ is satisfied for the vast majority of simple groups $L$, namely, for alternating groups of degree at least~$19$, classical groups of dimension at least~$11$, and all exceptional groups except $G_2(q)$, $^3D_4(q)$, and $^2F_4(2)'$ (see more in~\cite{11VasVd.t}).
\end{rem}

Despite the fact that there are a huge number of simple groups recognizable by spectrum (see \cite[Theorem 1]{21Survey_arxiv}), there are only two examples of a recognizable group that is a direct square of a simple group: $Sz(2^7)\times Sz(2^7)$ \cite{97Maz1.t} and $J_4\times J_4$ \cite{21GorMas}. The reason for this is that the standard methods of proving the nonsolvability of groups isospectral to a simple group are based on properties of its prime graph. Clearly, they are not applicable to proving the nonsolvability of a group isospectral to the square of a simple group, since in this case the corresponding graph is always complete. Proposition~1 of \cite{21GorMas}, mentioned in Remark 3, is a significant advance in this direction, but it is difficult to verify the conditions of this proposition for squares of arbitrary simple groups. As Corollary~\ref{c:sqr} shows, Theorem~\ref{th:criterion} of the present paper does not have this disadvantage and therefore allows us to hope for new examples of recognizable groups that are the squares of simple groups.

In this paper we consider the direct squares of simple Suzuki groups~$Sz(q)$. Recall that $Sz(q)={}^2B_2(q)$, where $q=2^\alpha$ with odd $\alpha$, is a twisted exceptional group of Lie type. It is well known that $Sz(q)$ is a nonabelian simple group if $q\geqslant 8$ and a Frobenius group of order $5\cdot 4$ if $q=2$. As mentioned above, $Sz(2^7)\times Sz(2^7)$ is uniquely (up to isomorphism) characterized by its spectrum in the class of finite groups~\cite{97Maz1.t}. We prove the following generalization of this result.

\begin{theorem}\label{th:suzuki}
If $q\geqslant 8$ and $q\neq 32$, then the group $Sz(q)\times Sz(q)$ is recognizable by spectrum.
\end{theorem}

\begin{theorem}\label{th:suzuki32}
Let $L=Sz(32)$. The finite groups isospectral to $L\times L$ are exactly 
the groups of the form $(L\times L)\rtimes \langle \varphi\rangle$, where either $\varphi=1$, or $\varphi$ normalizes each direct factor and induces  
a field automorphism of order $5$ on it. In particular, up to isomorphism,  there are four finite groups isospectral to $L\times L$.
\end{theorem}

Note that the recognizability by spectrum of the simple Suzuki groups themselves was established by W. Shi~\cite{92Shi} about 30 years ago. It is also worth noting that the proof of Theorem~\ref{th:criterion} does not depend on the classification of finite simple groups, while the proof of Theorems~\ref{th:suzuki} and \ref{th:suzuki32} uses only the assertion proved back in 1977 (see \cite{77Pod.t}, \cite{77Gor}, or \cite{77FSS}) that nonabelian simple groups whose orders are not divisible by $3$ are exactly the Suzuki groups.

The article is organized as follows. In the first section, the necessary preliminary information and results are collected, in the second, the nonsolvability criterion from Theorem~\ref{th:criterion} is proved, and in the third, the recognizability of the squares of Suzuki groups is established.

Concluding the introduction, the authors would like to express their deep gratitude and admiration to Viktor Danilovich Mazurov whose article~\cite{97Maz1.t}, published in {\em Algebra and Logic} 25 years ago, served as a source of inspiration for this paper. Also the authors are grateful to A.\,A.\,Buturlakin for valuable comments.

\section{Preliminaries}

For a natural number~$n$, $\pi(n)$ denotes the set of prime divisors of~$n$. If $\pi$ is the set of primes, then $\pi'$ is the set of primes not in~$\pi$.

Let $G$ be a group. The set $\omega(G)$ is closed with respect to taking divisors; therefore, it is uniquely defined by its subset $\mu(G)$ consisting of elements maximal with respect to divisibility relation. The exponent of $G$ is denoted by~$\gexp(G)$. Also, $\Aut G$ and $\Out G$ are respectively the automorphism group and the outer automorphism group of~$G$, $F(G)$ is its Fitting subgroup, and $\Soc(G)$ is its socle, i.e., the product of the minimal normal subgroups. If $G$ is a $p$-group, then $\Omega_1(G)$ is its subgroup generated by elements of order~$p$. If a group $B$ acts on a group $A$, then $A\rtimes B$ denotes their natural semidirect product.

\begin{lemma}[{Bang \cite{86Bang}, Zsigmondy \cite{Zs}}]\label{l:bangz} Let $q,n\geqslant 2$ be integers. Then either there is a prime number $r$ that divides $q^n-1$ and does not divide $q^i-1$ for all $i<n$, or one of the following conditions is satisfied:
\begin{enumerate}
 \item[(1)] $q=2$ and $n=6$;
 \item[(2)] $q$ is a Mersenne prime and $n=2$.
 \end{enumerate}
\end{lemma}

A prime number $r$ from Lemma~\ref{l:bangz} is called a \emph{primitive prime divisor} of $q^n-1$.

\begin{lemma}\label{l:3coclique}
If $G$ is a solvable group, then $t(G)\leqslant2$.
\end{lemma}

\begin{proof} Otherwise, the spectrum of a Hall $\sigma$-subgroup of~$G$, where $\sigma$ is the set of three vertices of the graph $GK(G)$ that are not pairwise adjacent to each other, consists only of powers of primes, which contradicts \cite[Theorem~1]{57Hig}.
\end{proof}


A group $G$ is called a {\em cover} of a group $L$ if there exists an epimorphism from $G$ onto~$L$. The following four lemmas give known facts about the orders of elements in covers of finite groups. Note that the references are given to those papers in which these statements were formulated in a form convenient for further discussion. Say,  Lemma~\ref{l:hh} is a variation of the celebrated Hall -- Higman theorem~\cite{56HalHig}.

\begin{lemma}{{\rm\cite[Lemma~10]{99ZavMaz.t}}}\label{l:split} Let $K$ be a normal elementary abelian subgroup of a group~$G$, $G/K\simeq L$ and $G_1=K\rtimes L$ is the natural semidirect product. Then $\omega(G_1)\subseteq\omega(G)$.
\end{lemma}

\begin{lemma}{{\rm\cite[Lemma~4]{97Maz1.t}}}\label{l:r^2s^2}
Let $G=P\rtimes (R_1\times \dots\times R_k)$, where $P$ is a nontrivial $p$-sub\-group, $R_i$ is an elementary abelian group of order~$r_i^2$, $1\leqslant i\leqslant k$, and primes $p$, $r_1$, \dots, $r_k$ are pairwise distinct. Then $G$ contains an element of order~$pr_1\dots r_k$.
\end{lemma}

\begin{lemma}{{\rm\cite[Lemma~1]{97Maz.t}}}\label{l:frob}
Let $P$ be a normal $p$-subgroup of a finite group $G$ and let $G/P$ be a Frobenius group with kernel $F$ and cyclic complement~$C$. If $p$ does not divide $F$ and $F\not\subseteq PC_G(P)/P$, then $G$ contains an element of order~$p|C|$.
 \end{lemma}

\begin{lemma}{{\rm\cite[Lemma~3.6]{15Vas}}}\label{l:hh}
Let $G=R\rtimes \langle g\rangle$, where $R$ is an $r$-group, $g$ is of prime order~$s$, $s\neq r$, and $[R,g]\neq1$. Suppose that $G$ acts faithfully on an elementary abelian $p$-group~$P$, where $p\neq r$.  Then either the natural semidirect product $P\rtimes G$ contains an element of order $ps$, or the following conditions are satisfied:
\begin{enumerate}
 \item[(1)] $C_R(g)\neq 1$;
 \item[(2)] $R$ is nonabelian;
 \item[(3)] $r=2$ and $s$ is a Fermat prime.
\end{enumerate}
\end{lemma}

Lemma~\ref{l:suz} collects well-known facts about Suzuki groups, and the next two lemmas deal with element orders in covers and automorphic extensions of these groups.

\begin{lemma} \label{l:suz} Let $G=Sz(q)$, where $q=2^\alpha\geqslant 8$. Then
\begin{enumerate}
 \item[(1)] $|G|=q^2(q-1)(q^2+1);$
 \item[(2)] $\mu(G)=\{4, q-1, q+\sqrt{2q}+1, q-\sqrt{2q}+1\};$
 \item[(3)] $G$ includes Frobenius groups of orders $q^2(q-1)$, $4(q+\sqrt{2q}+1)$ and $4(q-\sqrt{2q}+1)$ with cyclic complements of orders
  $(q-1)$, $4$ and $4$, respectively{\rm;}
 \item[(4)] $\Aut G=G\rtimes \langle \varphi\rangle$, where $\varphi$ is a field automorphism of $G$ of order~$\alpha$, and if $\psi\in\langle\varphi\rangle$, then $C_G(\psi)\simeq Sz(q^{1/|\psi|});$
 \item[(5)] the Schur multiplier of $G$ is trivial for $q>8$ and is an elementary abelian group of order $4$ for $q=8$.
\end{enumerate}
\end{lemma}

\begin{proof}
Items 1--4 were established in~\cite{62Suz}, Item~5 --- in~\cite{66AlpGor}.
\end{proof}

\begin{lemma}\label{l:SuzRep}
Let $G=Sz(q)$, where $q\geqslant 8$, and  $g\in G$ an $r$-element for an odd prime~$r$. If $G$ acts faithfully on a $p$-group $V$ and $p$ is odd, then the coset $Vg$ of the natural semidirect product $V\rtimes G$ contains an element of order~$p|g|$.
\end{lemma}

\begin{proof} We can assume that $V$ is an elementary abelian $p$-group and consider $V$ as a $G$-module. We may also assume that $G$ acts irreducibly on~$V$. By \cite[Theorem~1.1]{21TiepZal_arxiv}, the minimal polynomial of $g$ in this representation is equal to $x^{|g|}-1$. Hence, there exists $v\in V$ such that $v(1+g+g^2+\dots+g^{|g|-1})\neq 0$. Then $(vg)^{|g|}\neq 0$, hence $|vg|=p|g|$.
\end{proof}

\begin{lemma}\label{l:large_orders} Let $G=Sz(q)$, where $q\geqslant 8$.
\begin{enumerate}
\item[(1)] If $\psi$ is a field  automorphism of $G$, then the set of the orders of elements in the coset $G\psi$ is equal to $|\psi|\cdot \omega(Sz(q^{1/|\psi|}))$.
\item[(2)] If $g\in\Aut G$, then $|g|<2q$.
 \item[(3)] If $q\geqslant 32$ and $g\in\Aut G\setminus G$, then $|g|<q$.
 \end{enumerate}
\end{lemma}

\begin{proof}
Let $q=2^\alpha$. The set of element orders of $\Aut G\setminus G$ is the union of the sets $\gamma\cdot \omega(Sz(q^{1/\gamma}))$ over all nonidentity divisors $\gamma$ of  $\alpha$ \cite[Theorem 2]{17Gr.t}. It is clear from Item~2 of Lemma~\ref{l:suz} that each number in $\omega(Sz(q^{1/\gamma}))$ is less than $2q^{1/\gamma}$. Since $2\gamma q^{1/\gamma}\leqslant 2q$ for $q\geqslant 8$ and $2\gamma q^{1/\gamma}\leqslant 6q^{1/3}<q$ for $q \geqslant 32$, we obtain the required estimates.
\end{proof}

\section{Proof of Theorem~\ref{th:criterion}}

Let $G$ be a minimal (by order) counterexample to the assertion of the theorem. Since $G$ is a solvable group, Hall's theorem ensures the existence of a Hall $\sigma$\nobreakdash-\hspace{0pt}subgroup for any $\sigma\subseteq\pi(G)$. If we take $\sigma$ consisting of four primes from the subset $\sigma(G)$ given in the theorem, then the corresponding Hall $\sigma$-subgroup is solvable and satisfies all conditions of the theorem, therefore, by the assumption on the minimality of counterexample, it must coincide with~$G$.

Thus, in what follows, $G$ is a $4$-primary solvable group satisfying the hypotheses of the theorem, i.e., $\pi(G)=\{p,q,r,s\}$, and the spectrum of $G$ contains all pairwise products of numbers from $\pi(G)$ (Condition 1 of the theorem), but does not contain any one of the triple products (Condition 2).

\begin{lemma}\label{l:pi4}
For each nontrivial element $x\in G$, the inequality $|\pi(C_G(x))|<4$ holds.
\end{lemma}

\begin{proof}
Suppose that $G$ has a nontrivial element $x$ with $|\pi(C_G(x))|=4$. We may assume that $x$ has prime order. By Lemma~\ref{l:3coclique}, $C_G(x)$ contains an element whose order is divisible by three different primes from $\pi(G)$, and this contradicts Condition 2 of the theorem.
\end{proof}

\begin{lemma}\label{l:p^2exists}
If $p\in\pi(G)$, then $G$ contains a noncyclic subgroup of order~$p^2$.
\end{lemma}

\begin{proof}
Otherwise, $G$ contains, up to conjugation, a unique subgroup $P$ of order~$p$. Condition~1 of the theorem yields $|\pi(C_G(P))|=4$, where $C_G(P)$ is the centralizer of $P$ in~$G$, which contradicts Lemma~\ref{l:pi4}.
\end{proof}

\begin{lemma}\label{l:noncyclic}
Let $H$ be a Hall subgroup of $G$. Then the Fitting subgroup $F(H)$ of $H$ is not cyclic and $|\pi(F(H))|\leqslant2$.
\end{lemma}

\begin{proof}
Denote $F(H)$ by~$F$. The inequality $|\pi(F)|\leqslant2$ follows from Condition 2 of the theorem. 

Assume that $F$ is cyclic. Then $H\neq F$ by Lemma \ref{l:p^2exists}. Since the group $H$ is solvable, we have  $C_H(F)\subseteq F$ \cite[Theorem 6.1.3]{68Gor}. Therefore, the quotient group $H/F$ embeds in the group $\Aut F$ and, in particular, is abelian. If $F$ is a $p$-group, then by Lemma~\ref{l:p^2exists}, $H/F$ contains a noncyclic $p'$-group, which is impossible because the $p'$-part of $\Aut F$ is cyclic. Let $\pi(F)=\{p,q\}$, where $p>q$, and let $W$ be a Sylow $q$-subgroup of~$F$. Since $p>q$, a Sylow $p$-subgroup $P$ of $H$ centralizes~$W$. The group $H/F$ is abelian, so the image of $P$ in $H/F$ is normal in~$H/F$, and hence $P\times W$ is normal in~$H$. It follows that $P\leq F$, which contradicts Lemma~\ref{l:p^2exists}.
\end{proof}

\begin{lemma}\label{l:minSyl}
Every minimal normal subgroup of $G$ is a Sylow subgroup~of~$G$.
\end{lemma}

\begin{proof}
Let $N$ be a minimal normal $p$-subgroup of~$G$, and let $H$ be a Hall $p^{\prime}$-subgroup of~$G$. By Lemma~\ref{l:p^2exists}, the group $H$ contains an elementary abelian $q$-subgroup of order $q^2$ for every $q\in\pi(H)$, and by Lemma~\ref {l:r^2s^2}, the subgroup $NH$ contains an element of order~$pq$. Thus, the group $NH$ satisfies the hypotheses of the theorem, and so $G=NH$.
\end{proof}

We fix until the end of the proof the notation: $F$ is the Fitting subgroup $F(G)$ of~$G$ and $H$ is a complement of $F$ in~$G$, i.e., a Hall $\pi(F)'$-subgroup of~$G$.
By Lemma~\ref{l:minSyl}, the group $F$ is a direct product of Sylow noncyclic elementary abelian subgroups of~$G$.

\begin{lemma}\label{l:normH}
Let $N$ be a normal $r$-subgroup of~$H$. Then either $N$ is a Sylow $r$-subgroup of~$G$ and $\gexp(N)=r$, or $N$ is a cyclic group of odd order, or $N$ has a characteristic subgroup of order~$2$.
\end{lemma}

\begin{proof}
If $N$ is a cyclic group of even order or $N$ is not cyclic but does not contain elementary abelian subgroups of order $r^2$, then $r=2$ and $N$ has a characteristic subgroup of order~$2$. Hence, we may assume that $N$ includes an elementary abelian subgroup of order~$r^2$. By Lemma~\ref{l:r^2s^2}, the subgroup $FN$ contains an element of order $rp$ for every $p\in \pi(F)$. Arguing further as in Lemma~\ref{l:minSyl}, we obtain that $N$ is a Sylow subgroup of~$H$, and so of~$G$.

By the minimality of~$G$, there are no proper noncyclic characteristic abelian subgroups in~$N$. If $N$ is abelian, then $N=\Omega_1(N)$ and $\gexp(N)=r$, as required. Let $N$ be nonabelian. Then $Z(N)$ is a cyclic group, and for $r=2$, $N$ has a characteristic subgroup of order $2$. If $r\neq 2$, then in virtue of \cite[Theorem~5.5.3]{68Gor} the group $N$ is the central product of a cyclic group and an extraspecial group $M$ of period $r$. Then $N=\Omega_1(N)=M$, and we are done.
\end{proof}

Recall that  $|\pi(F)|\leqslant2$ by Lemma \ref{l:noncyclic}. Therefore, Theorem~\ref{th:criterion} will follow from the next two lemmas.

\begin{lemma}\label{l:Fneq1} $|\pi(F)|\neq1$.
\end{lemma}

\begin{proof}
Assume the opposite. Then $F$ is a noncyclic elementary abelian $p$-group which is a Sylow subgroup of~$G$, and $H$ is a Hall $p^\prime$-subgroup of~$G$.

Let $N$ be a normal primary subgroup of $H$. If $N$ has a characteristic subgroup of order~$2$, then this subgroup is central in~$H$ and, therefore, $G$ contains an element of order $2rs$ forbidden by Condition 2, where $\{r,s\}=\pi(G)\setminus\{2,p\}$. Applying Lemma~\ref{l:normH}, we obtain that either $N$ is a Sylow subgroup of $G$ of prime exponent, or $N$ is a cyclic group of odd order.

Consider now the Fitting subgroup $U=F(H)$ of~$H$. By Lemma~\ref{l:noncyclic}, the group $U$ is not cyclic. Hence, one of the Sylow subgroups of $U$, say, a $q$-subgroup, is not cyclic. Thus, by the previous paragraph, $U$ contains a Sylow $q$-subgroup $Q$ of $G$ and $Q$ is of prime exponent. Below we consider separately the cases when $|\pi(U)|=2$ and $|\pi(U)|=1$.

Let $\pi(U)=\{q,r\}$ and $V$ be a Sylow $r$-subgroup of~$U$. Then $U=Q\times V$. If $V$ is also a Sylow subgroup of $G$, then, applying Lemma~\ref{l:r^2s^2}, we see that $pqr\in\omega(G)$. Hence $V$ is a cyclic group and $r\neq2$.

Let $S$ be a Sylow $s$-subgroup of~$G$. If $C_S(V)$ contains a noncyclic elementary abelian subgroup, then $C_S(V)$ contains an element that centralizes nontrivial elements in both $Q$ and $V$, and so $H$ contains an element of forbidden order $qrs$. Therefore, $C_S(V)$ has at most one subgroup of order~$s$. In particular, $C_S(V)<S$ due to Lemma~\ref{l:p^2exists}. On the other hand, $S/C_S(V)=N_S(V)/C_S(V)$ is a cyclic group of order dividing $r-1$. Hence $C_S(V)\neq 1$ and $s<r$. Denote by $Z$ the unique subgroup of order $s$ in~$C_S(V)$. It is clear that $Z\unlhd S$ and, therefore, $Z\leq Z(S)$.

Let $J$ be a Hall $\{r,s\}$-subgroup of $H$ containing $S$ and $\ov{J}=J/V$. The images of other subgroups of $J$ in $\ov J$ will also be denoted by a bar.

Let $\ov K=F(\ov{J})$ and $\ov C=\ov{C_S(V)}\cap \ov K$. Then $\ov C$ is normal in $\ov J$, because $$\ov C=\ov{C_{J}(V)}\cap\left(\ov S\cap\ov K\right),$$ while $\ov{C_{J}(V)}$ and $\ov S\cap\ov K$ are normal subgroups of~$\ov J$. Assume that $\ov Z\not\leq \ov K$. Since $C_{\ov J}(\ov K)\leq \ov K$, $\ov K$ must contain a Sylow $r$-subgroup~$\ov T$ on which $\ov Z$ acts nontrivial. Applying Lemma~\ref{l:hh} to the action of $\ov T\rtimes \ov Z$ on $Q$ and noting that $r\neq2$, we obtain an element of order $sq$ in~$QZ$. However, $QZ$ centralizes $V$, so $G$ contains an element of forbidden order $sqr$. Thus, $\ov Z\leq \ov K$, whence $\ov Z\leq \ov C$, and $\ov Z$ is normal in $\ov J$ as the only subgroup of $\ov C$ of order~$s$.

Let $R$ be a Sylow $r$-subgroup of $J$ (and~$G$). Then $[\ov R,\ov Z]=1$ since $r>s$. Therefore, $Z$ stabilizes the normal series $1<V<R$ and so acts trivially on~$R$ \cite[Lemma 5.3.2]{68Gor}. Consider the action of $R\times Z$ on~$QF$. The group $Z$ acts on $Q$ without fixed points, otherwise $G$ contains an element of order $sqr$. Hence $QZ$ is a Frobenius group, and since $C_G(F)=F$, it acts on $F$ faithfully. Therefore, $C_F(Z)\neq 1$ in view of Lemma~\ref{l:frob}. The group $R$ acts on $C_F(Z)$ and, applying Lemmas \ref{l:p^2exists} and \ref{l:r^2s^2}, we obtain an element of order $rpq$; a contradiction.

Let now $\pi(U)=\{q\}$, i.e., $U=Q$ is a Sylow $q$-subgroup of~$G$. Let, as above, $J$ be a Hall $\{r,s\}$-subgroup of~$H$. Then $G=FUJ$.

Put $V=F(J)$. By Lemma~\ref{l:noncyclic}, the group $V$ is not cyclic. Hence one of its Sylow subgroups, say, the $r$-subgroup $W$ is not cyclic. Denote by $S$ a Sylow $s$-subgroup of~$J$ (and~$G$).

Assume that $r=2$ and $W$ contains a unique involution~$z$. Then the whole group $S$ centralizes~$z$. If $C_U(z)=1$, then $U\rtimes\langle z\rangle$~ is a Frobenius group and
$C_F(z)\neq 1$ by Lemma~\ref{l:frob}. Hence, at least one of the centralizers $C_U(z)$ and $C_F(z)$ is nontrivial. This centralizer is invariant under the action of $S$, so applying Lemma~\ref{l:r^2s^2}, we obtain an element of order $2sq$ or $2sp$; a contradiction. Arguing in the same way as in Lemma~\ref{l:normH}, we conclude that $W$ is a Sylow $r$-subgroup of $G$ and either $W$ is abelian or has an odd order.

Let $\pi(V)=\{r,s\}$ and $T$ a Sylow $s$-subgroup of~$V$, i.e., $V=W\times T$. Repeating the argument of the previous paragraph about the action of $S\times\langle z\rangle$ on~$UF$ with $W$ instead of~$S$ and any element $x\in T$ instead of~$z$, we arrive at a contradiction.

Thus, $V=W$. There is a nontrivial element $x\in S$ of order $s$ such that $C_V(x)\neq 1$. On the other hand, $[V,x]\neq 1$. The group $V\rtimes\langle x\rangle$ acts faithfully on $U$ and on~$F$. By Lemma \ref{l:hh}, we get that $C_U(x)\neq 1$ and $C_F(x)\neq 1$. Hence $|\pi(C_G(x))|=4$. This contradiction completes the proof of the lemma.
\end{proof}

\begin{lemma}\label{l:Fneq2} $|\pi(F)|\neq2$.
\end{lemma}

\begin{proof}
Again, suppose the contrary. Then $F=P\times Q$, where $P$ and $Q$ are noncyclic elementary abelian Sylow $p$-subgroups and $q$-subgroups of~$G$, and $H$ is a Hall $\{r,s\}$-subgroup of~$G$.

Assume first that $H$ is of odd order, and let $U=F(H)$ be the Fitting subgroup of $H$. By Lemma~\ref{l:noncyclic}, the group $U$ is not cyclic, therefore, applying Lemma~\ref{l:normH}, we obtain that $U$ contains a Sylow subgroup of~$G$, say, an $r$-subgroup $R$, and $R$ has exponent~$r$. Denote by $S$ a Sylow $s$-subgroup of~$H$ (and~$G$). By the minimality of~$G$, the group $S$ is an elementary abelian group of order~$s^2$.

By Lemma~\ref{l:r^2s^2}, there are $x,y\in S$ such that $C_P(x)\neq1$ and $C_Q(y)\neq1$. Condition 2 implies that $C_Q(x)=C_P(y)=1$. In particular, $x\neq y$. If one of these elements, say $x$, centralizes $R$, then $R$ acts on $C_P(x)$ and we get an element of order~$prs$. Hence $[R,x]\neq 1$ and $[R,y]\neq1$.

If $[R,x]\cap C_R(Q)=1$, then by Lemma~\ref{l:hh}, we have $C_Q(x)\neq 1$, but this is not the case. In particular, $C_R(Q)>1$.
The group $C_R(Q)$ is cyclic, otherwise by Lemma~\ref{l:r^2s^2}, $G$ contains an element of order~$rqp$. Thus, $C_R(Q)$ is of order $r$ and $[R,x]=C_R(Q)$. Similarly, $[R,y]=C_R(P)$. It is clear that $C_R(Q)\cap C_R(P)=C_R(F)=1$, so $R$ includes a subgroup $[R,x]\times [R,y]$ of order~$r^2$. Let $g,h\in R$ be such that $[g,x]\neq1$, $[h,y]\neq1$. Then $[g,xy]=[g,y][g,x]^y\not\in [R,y]$ and $[h,xy]=[h,y][h,x]^y\not\in [R,x]$. Consequently, $[R,xy]$ lies neither in $C_R(P)$ nor in $C_R(Q)$, and hence there is an element of order $pqs$ in $G$; a contradiction.

Now let $2\in\pi(H)$ and let $s$ be the only odd prime divisor of the order of~$H$. By Lemma~\ref{l:p^2exists}, the group $H$ contains a noncyclic subgroup $T$ of order~$4$. If $t$  is an involution from $T$, then Condition 2 of the theorem yields $C_P(t)=1$ or $C_Q(t)=1$. Choose from $P$ and $Q$ the group on which two different involutions $t$ and $u$ of $T$ act fixed-point-freely. Let this be~$P$. Then the involution $v=tu$ centralizes $P$ and inverts~$Q$. The involution $v^h$ has the same property for every $h\in H$. If $v^h\neq v$, we get an element of forbidden order in $G$: $2pq$ in the case when the dihedral group $D=\langle v, v^h\rangle$ contains an elementary abelian subgroup of order $4$, and $spq$ otherwise. Thus, the involution $v$ lies in the center of~$H$. The group $PH$ contains an element $x$ of order~$ps$, so $xv$ is an element of order~$2ps$. This contradiction proves the lemma, and the theorem as well.
\end{proof}

\section{Recognition of squares of Suzuki groups}

In this section, we will prove Theorems~\ref{th:suzuki} and~\ref{th:suzuki32}. We will use the main properties of Suzuki groups listed in Lemma~\ref{l:suz} without explicit reference to this lemma.

Let $L=Sz(q)$, where $q=2^{\alpha}\geqslant 8$. The set $\mu(L)$ consists of the numbers $m_1(q)=4$, $m_2(q)=q-1$, $m_3(q)=q-\sqrt{2q}+1$, and $m_4(q)=q+\sqrt{2q}+1$. Note that $m_3(q)m_4(q)=q^2+1$, $q/2<m_i(q)<2q$ for $i\neq1$, and $(m_i(q),m_j(q))=1$ for $i\neq j$. In particular, $t(L)=4$.

Let $G$ be a finite group with $\omega(G)=\omega(L\times L)$. Then $$\mu(G)=\{m_i(q)\cdot m_j(q) \mid 1\leqslant i, j\leqslant 4 \text{ and } i\neq j\}.$$
By Corollary~\ref{c:sqr}, the group $G$ is nonsolvable. Let $K$ be the solvable radical of~$G$ and $\overline G=G/K$. Since $3\not\in\pi(G)$, it follows  that $\Soc(\overline G)=L_1\times L_2\times\dots\times L_k$, where $L_i=Sz(2^{\alpha_i})$ and $\alpha_i\geqslant 3$ for all $1\leqslant i\leqslant k$.

For every $i$, the number $2^{\alpha_i}-1$ must divide the exponent of~$G$, which in turn divides $4(2^{4\alpha}-1)$, so all $\alpha_i$ divide $\alpha$. Hence $m_2(2^{\alpha_i})$ divides $m_2(q)$. Furthermore, one of the numbers $m_3(2^{\alpha_i})$ and $m_4(2^{\alpha_i})$ divides $m_3(q)$, and the other divides $m_4(q)$ (see, e.g., \cite[p.~18]{99Luc}). This immediately implies that $k\leqslant 2$: otherwise, $4ps\in\omega(G)\setminus\omega(L\times L)$ for some prime divisors $p$ and $s$ of $m_2(2^{\alpha_2})$ and $m_3(2^{\alpha_3})$.

In what follows, we fix some primitive prime divisors $r_2$ and $r_4$ of the numbers $2^{\alpha}-1$ and $2^{4\alpha}-1$, respectively, and note that each of them is greater than~$\alpha$. In particular, none of them divides $|\Aut L_i/L_i|$ for any~$i$.

Let $k=1$. Suppose $\alpha_1\neq\alpha$. Then $r_2$ and $r_4$ do not divide $|\overline G|$ and so $r_2r_4\in\omega(K)$. It is clear that $r_4$ divides $m_i(q)$ for $i\in\{3,4\}$. Let $j\in\{3,4\}$ and $j\neq i$. Since $m_j(q)>q/2=2^{\alpha-1}\geqslant 2^{\alpha_1+1}$, it follows from Item 2 of Lemma~\ref{l:large_orders} that $m_j(q)\not\in\omega(\overline G)$. On the other hand, $r_2m_j(q), r_4m_j(q)\in\omega(G)$, so $r_2s, r_4s\in\omega(K)$ for some prime divisor $s$ of $m_j(q)/(m_j(q),\exp(\overline G))$. Similarly, if $T$ is a Sylow $2$-subgroup of~$G$, then $2s\in\omega(KT)$. Thus, the solvable group $KT$ satisfies the hypothesis of Theorem~\ref{th:criterion} with $\sigma(KT)=\{r_2,r_4, 2,s\}$; a contradiction.

Therefore, $\alpha_1=\alpha$ and $L_1=L$. Let $\alpha\geqslant5$. Denote by $G_1$ the preimage of $L$ in $G$ and show by induction on the order of~$K$ that $L$ is a direct factor in~$G_1$. Let $V$ be a minimal normal in $G_1$ subgroup of~$K$. Then $V$ is an elementary abelian $p$-group for some prime~$p$. By induction, we may assume that $G_1/V=K/V\times H/V$, where $H/V\simeq L$. Suppose $C_H(V)=V$. If $p=2$, then applying Lemma~\ref{l:frob} to a Frobenius subgroup of order $4m_3(q)$, we obtain that $8\in\omega(G)$. If $p$ is odd, then by Lemma~\ref{l:SuzRep}, $G$ contains an element of order $p^{l+1}$, where $p^l$ is the greatest power of~$p$ in~$\omega(L)$. Therefore, $C_H(V)=H$. In view of $\alpha\geqslant 5$, the Schur multiplier of~$L$ is trivial, so $H=V\times L$. Since $[K,L]\leq V$, the equalities $[K,L,L]=1$ and $[L,K,L]=1$ hold, and hence $[L,K]=[L ,L,K]=1$.

Thus, $K\times L=G_1\leq G$ and hence $\omega(K)\subseteq\omega(L)$. As above, choose $j\in\{3,4\}$ so that $r_4$ does not divide $m_j(q)$, and let $s\in\pi(m_j(q))$.
By virtue of Lemma~\ref{l:3coclique} and the inclusion $\omega(K)\subseteq\omega(L)$, the intersection $\pi(K)\cap\{2,r_2,r_4,s\}$ does not contain more than two numbers.
On the other hand, the numbers $2$, $r_2$, $r_4$ do not divide $|G/(K\times L)|$, so two of them must divide $|K|$. Denote by $p$ the third one. Then $pm_j(q)$ is coprime to~$|K|$, so $pm_j(q)\in\omega(\overline G)$. Applying Item~3 of Lemma~\ref{l:large_orders}, we get that $q>pm_j(q)>pq/2$; a contradiction.

Let $\alpha=3$. Then the Schur multiplier of~$L$ has order~$4$, $m_2(q)=7$, $m_3(q)=5$, and $m_4(q)=13$. It is clear that $\overline G=\Soc(\overline G)$. Suppose $7\not\in\pi(K)$ and  let $x\in G$ be an element of order~$7$. Since $\langle x\rangle$ is a Sylow subgroup of $G$, we have $\pi(C_K(x))=\{2,5,13\}$. By Lemma~\ref{l:3coclique}, $C_K(x)$ contains an element of order $pr$ for $p\neq r$ and  hence $7pr\in\omega(G)$; a contradiction. Hence $7\in\pi(K)$. Similarly, $5,13\in\pi(K)$.

Arguing as in the case $\alpha\geqslant 5$, we conclude that either $L$ is again a direct factor, i.e., in this case $G=K\times L$, or there is a normal solvable subgroup $W$ of $G$ such that the quotient group $G/W$ is the central product of a solvable group $N$ and a perfect central extension $M$ of $L$, and the center $Z(M)$ is nontrivial. In the first case, applying Lemma~\ref{l:3coclique}, we see that the product of two of the numbers $5,7,13$ lies in $\omega(K)$, and hence the product of all three numbers lies in $\omega(G)$; a contradiction. 

Let us consider the second case. Since $Z(M)\neq1$, there are no odd numbers in $\pi(N)$; in particular, $7\in\pi(W)$. Denote by $J$ the preimage of $M$ in $G$. By the Frattini argument, we may assume that a Sylow $7$-subgroup $P$ of $W$ is normal in~$J$. Then replacing $G$ by $G/\Phi(P)$, we may assume that $P$ is elementary abelian. Finally, by Lemma \ref{l:split} we may assume that $J$ is a split extension of $P$ by $\widetilde J\simeq J/P$. If $C_J(P)=J$, then $2\cdot 5\cdot 7\in\omega(G)$. 

Let $C_J(P)W/W\leqslant Z(M)$ and let $R$ be a Sylow 2-subgroup of $\widetilde J$.  By the Frattini argument, there is an element $g\in N_{\widetilde J}(R)$ of order $7$, and let $T=R\rtimes \langle g\rangle$. Since $C_J(P)W/W\leqslant Z(M)$, it follows that there is a homomorphism from $T/C_P(T)$ 
onto a Borel subgroup of $L$, which is a Frobenius group with kernel of order $2^6$ and complement of order~$7$. Hence an element of $T/C_T(P)$ of order 7 acts on the Sylow 2-subgroup of $T/C_T(P)$ nontrivially. Applying Lemma~\ref{l:hh} to the action of $T/C_T(P)$ on $P$ and noting that $7$ is not a Fermat prime, we obtain an element of order $49$ in~$G$. This contradiction completes the proof that $k\neq 1$. 

Thus $k=2$. Suppose $K\neq1$ and show that $\omega(G)\not\subseteq\omega(L\times L)$. Without loss of generality, we may assume that $K$ is an elementary abelian $p$-group. By Lemma~\ref{l:split}, we may also assume that $L_1\times L_2$ is a subgroup of~$\ov{G}$. Suppose that $p\neq 2$ or $L_1\times L_2$ centralizes $K$. There are odd $r\in\pi(L_1)$ and $s\in\pi(L_2)$ such that $prs\not\in\omega(L\times L)$, and let $x\in L_1$ and $y\in L_2$ be elements of orders $r$ and~$s$, respectively. Applying Lemma~\ref{l:SuzRep} to $L_1$, we get that $C_K(x)\neq~1$. The group $L_2$ acts on $C_K(x)$, and applying Lemma~\ref{l:SuzRep} again, we see that $y$ has a fixed point in~$C_K(x)$. Thus, $prs\in\omega(G)$. Hence $p=2$ and, say, $L_1$ acts on $K$ faithfully. Since $L_1$ has Frobenius subgroups with kernels of orders $m_3(q)$, $m_4(q)$ and cyclic complements of order $4$, it follows by Lemma~\ref{l:frob} that $G$ contains an element of order~$8$.

Therefore, $K=1$. If at least one of the numbers $\alpha_1$ and $\alpha_2$ is distinct from~$\alpha$, then $r_2r_4\not\in\omega(G)$. Hence, $L_1\simeq L_2\simeq L$ and $\Aut(L_1\times L_2)=(L_1\times L_2)\rtimes (\langle \varphi\rangle\wr\langle \tau\rangle)$, where $\varphi$ is some fixed automorphism of $L_1$ of odd order $\alpha$ and $\tau$ is the automorphism of  $L_1\times L_2$ of order $2$ interchanging the components. If $2\in\pi(G/\Soc(G))$,
then we may assume that $\tau\in G$, and taking $g$ to be an element of $L_1$ of order $4$, we see that $8=|g\tau|\in\omega(G)$. This is false, and so $G\leq \Aut L_1\times \Aut L_2$. 

Suppose $G\neq L_1\times L_2$. 
Then $G$ contains $\psi=(\psi_1,\psi_2)$, where $\psi_i$ is a field automorphism of $L_i$, $i=1,2$,
and $|\psi|=p$ is a prime. It is clear that $$C_{L\times L}(\psi)=C_{L_1}(\psi_1)\times C_{L_2}(\psi_2)\simeq Sz(q^{1/|\psi_1|})\times Sz(q^{1/|\psi_2|}).$$  Assume that $q^{1/|\psi_i|}>2$ for some $i$. Then there is $s\in\pi(C_{L_i}(\psi_i))$ such   $s\neq 2$, $s\neq p$ and $ps\not\in\omega(L)$. For this prime $s$, we have $2ps\in\omega(G)\setminus\omega(L\times L)$, a contradiction. It follows that $|\psi_1|=|\psi_2|=p$ and $q=2^p$. On the other hand, $p\in\pi(G)=\pi(L)$, and so $p$ divides $(2^{p}-1)(2^{2p}+1)$. This yields $p=5$. Thus if $q\neq 2^5$, then $G\simeq L\times L$, and this completes the proof of Theorem \ref{th:suzuki}.

Let $q=2^5$ and $M=(L_1\times L_2)\rtimes \langle \psi\rangle$, where  $\psi$ induces on each factor $L_i$  a field automorphism of order $5$. We claim that $\omega(M)=\omega(L\times L)$. 
Let $g\in M\setminus(L_1\times L_2)$. Then $g=(g_1\psi_1,g_2\psi_2)$, where 
$|\psi_1|=|\psi_2|=5$. By Item 1 of Lemma \ref{l:large_orders}, for every $i=1,2$,  the order of $g_i\psi_i$ lies in $5\cdot \omega(Sz(2))$ and hence divides $20$ or $25$. Since the order of $g$ is equal to the least common multiple of $|g_1\psi_1|$ and $|g_2\psi_2|$, this order divides $100$ and, therefore, lies in $\omega(L\times L)$.

It remains to find the number of pairwise nonisomorphic groups $M$. Groups $A$ and $B$ such that $L_1\times L_2\leq A,B\leq \Aut(L_1\times L_2)$ are isomorphic if and only if their images in $\Out(L_1\times L_2)$ are conjugate. 
Recall that $\Aut(L_1\times L_2)=(L_1\times L_2)\rtimes (\langle \varphi\rangle\wr\langle \tau\rangle)$.  We denote the images of $\varphi$ and $\tau$ in $\Out(L_1\times L_2)$ by the same symbols. Then the image of $M$ in $\Out(L_1\times L_2)$ is the cyclic group $X_l$ generated by $(\varphi^l,\varphi^\tau)$ for some $1\leqslant l\leqslant 4$. It is clear that $X_l$ and $X_m$, with $l\neq m$, are conjugate if and only if $X_l^\tau=X_m$. It follows that $X_2$ and $X_3$ are conjugate, while $X_1$, $X_2$ and $X_4$ are pairwise not conjugate. Thus, up to isomorphism, there are four groups isospectral to $L\times L$: $L_1\times L_2$ and the full preimages of $X_l$, with $l=1,2,4$, in  $\Aut(L_1\times L_2)$. This completes the proof of Theorem \ref{th:suzuki32}.

~

Zhigang Wang,

School of Science, Hainan University, Haikou, Hainan, P. R. China;

wzhigang@hainanu.edu.cn

\medskip

Andrey V. Vasil'ev, Sobolev Institute of Mathematics,  Novosibirsk, Russia;

School of Science, Hainan University, Haikou, Hainan, P. R. China;

vasand@math.nsc.ru

\medskip

Maria A. Grechkoseeva,

Sobolev Institute of Mathematics,  Novosibirsk, Russia;

grechkoseeva@gmail.com

\medskip

Archil Kh. Zhurtov,

Kabardino-Balkarian State University, Nalchik, Russia;

zhurtov\!\_a@mail.ru

\end{document}